\documentclass[a4paper,11pt]{article}
\usepackage[plainpages=false]{hyperref}
\usepackage{amsfonts,latexsym,rawfonts,amsmath,amssymb,amsthm,mathrsfs}
\usepackage{amsmath,amssymb,amsfonts,latexsym,lscape,rawfonts}
\usepackage[usenames]{color}

\usepackage[all]{xy}
\usepackage{eufrak}
\usepackage{makeidx}         
\usepackage{graphicx,psfrag}

\usepackage{array,tabularx}

\usepackage{setspace}

\title{Stability of Calabi flow near an extremal metric}
\author{Hongnian Huang and Kai Zheng}
\date{}
\usepackage{amsmath,amsthm}
\usepackage{amsmath,amssymb,amsfonts,amsthm,graphics}
\newtheorem{thm}{Theorem}[section]
\newtheorem{cor}[thm]{Corollary}
\newtheorem{lem}[thm]{Lemma}

\newtheorem{rmk}[thm]{Remark}

\theoremstyle{definition}

\theoremstyle{remark}

\newcommand{\thmref}[1]{Theorem~\ref{#1}}

\makeatother
\begin{document}
\maketitle
\renewcommand{\sectionmark}[1]{}

\begin{abstract}
We prove that on a K\"ahler manifold admitting an extremal metric $\omega$ and for any K\"ahler potential $\varphi_0$ close to $\omega$, the Calabi flow starting at $\varphi_0$ exists for all time and 
the modified Calabi flow starting at $\varphi_0$ will always be close to $\omega$.
Furthermore, when the initial data is invariant under the maximal compact subgroup of the identity component of the reduced automorphism group, the modified Calabi flow converges to an extremal metric near $\omega$ exponentially fast.
\end{abstract}
\section{Introduction}
Let $M$ be a K\"ahler manifold and $\Omega$ be the K\"ahler class in $H^2(M, R) \cap H^{1,1}(M,C)$. By the $\partial\bar\partial$-lemma, any K\"ahler metric $\omega_{\varphi}$ in $\Omega$
can be written as
\begin{align*}
\omega_\varphi=\omega+\sqrt{-1}\partial\bar\partial\varphi
\end{align*}
for some smooth real valued K\"ahler potential $\varphi$.
The space of K\"ahler metrics is defined by
\[
\mathcal{H}
=\{\varphi\in C^{\infty}(M,R)\vert\omega
+\sqrt{-1}\partial\bar\partial\varphi>0 \}.
\]
Donaldson \cite{MR1736211}, Mabuchi \cite{MR909015} and Semmes \cite{MR1165352} independently
defined a Weil-Peterson typed metric on $\mathcal{H}$,
under which it becomes a non-positively curved
infinite dimensional symmetric space.
Chen \cite{MR1863016}
proved that any two points in $\mathcal{H}$
can be connected by a $C^{1,1}$ geodesic
and that $\mathcal{H}$ is a metric space,
which verify two of the Donalson's conjectures.

In order to tackle the existence of constant scalar curvature K\"ahler metric (cscK)
problem, Calabi \cite{MR645743}\cite{MR780039} introduced a well-known functional
\[
Ca(\varphi)=\int_M S(\varphi)^2 \ \omega^n_\varphi,
\]
where $S(\varphi)$ is the scalar curvature of $\omega_\varphi$.
The critical point of this Calabi functional is called an extremal K\"ahler metric.
Calabi discovered that an extremal K\"ahler metric is a cscK
if and only if the Calabi-Futaki invariant is equal to zero.
Later, he suggested that one may study the gradient flow of the K-energy
to search for the cscK. This flow is defined as
\begin{align}\label{calabi flow}
\frac{\partial\varphi}{\partial t}=S-\underline{S}
\end{align} and it decreases the Calabi energy.
Since it is a fourth order equation, the maximal principle fails.
In \cite{MR2103718}, Donaldson proposed a programme
to study the convergence of the Calabi flow.
On a Riemannian surface, P. Chru\'{s}ciel \cite{MR1101689}
proved that the flow exists for all time and converges to a cscK metric by using the Bondi mass. Later Chen \cite{MR1820328}  and Struwe \cite{MR1991140} gave a different proof assuming the uniformization theorem. In Chen-Zhu \cite{chenzhu}, they removed the assumption of the uniformization theorem.
For higher dimensions, the Calabi flow has been studied in
Chen-He \cite{MR2405167} \cite{chen-2007} \cite{chen-2008} and Tosatti-Weinkove \cite{MR2357473}.
In Chen-He \cite{MR2405167}, they proved that the Calabi flow can
start from a $C^{3,\alpha}$ K\"ahler potential
and become smooth immediately as $t>0$.
One defines the little H\"older space $c^{k,\alpha}$ to be the closure of smooth functions in the usual H\"older norm $C^{k,\alpha}$.

\begin{thm}\label{short time: stability}
If $\omega_{\varphi_0}=\omega+\sqrt{-1}\partial\bar\partial\varphi_0$
satisfies $|\varphi_0|_{c^{3,\alpha}(M,g)}\leq K$, and $\lambda \omega < \omega_0 =\omega_{\varphi_0} < \Lambda \omega$, where $K, \lambda, \Lambda$ are positive constants, then the Calabi flow initiating from $\varphi_0$ admits a unique solution
$$\varphi(t)\in C([0,T],c^{3,\alpha}(M,g))\cap C((0,T],c^{4,\alpha}(M,g))$$
for small $T=T(\lambda,\Lambda,K,\omega)$. More specifically, for any $t \in (0, T]$, there is a constant $C=C(\lambda,\Lambda,K,\omega)$ such that 
$$
t^{1/4} (|\dot{\varphi}(t)|_{c^{0,\alpha}}(M) + |\varphi(t)|_{c^{4,\alpha}(M)}) \leq C, |\varphi(t)|_{c^{3,\alpha}(M)} \leq C.
$$
\end{thm}

\begin{rmk}
He \cite{he-2009} shows that the Calabi flow can start from $\omega_\varphi$ where $\varphi \in c^{2,\alpha}(M)$.
\end{rmk}

\begin{thm}\label{regularity of PCF: reg of PCF}
The solution obtained above belongs to
$C^0([0,T],c^{3,\alpha}(M))\cap C^0((0,T],C^{\infty}(M))$.
\end{thm}

Chen and He then use an energy argument to show that
when K\"ahler manifolds admits a cscK $\omega$
and the initial K\"ahler potential is $C^{3,\alpha}$ small,
the Calabi flow converges exponentially fast to a cscK nearby.
In He \cite{he-2009}, he improved this result
for $C^{2,\alpha}$ small initial K\"ahler potentials.

In this short note,
we prove a parallel theorem for extremal K\"ahler metrics
using a different method
from Chen-Ding-Zheng \cite{zheng-2009}.
In that paper, they defined a flow called the pseudo-Calabi flow.
They proved the short time existence from $c^{2,\alpha}$
initial K\"ahler potentials, the long time existence under uniform Ricci bounds
and the stability near a cscK.
Since the linearized operator of the pseudo-Calabi flow
is not self-adjoint, they set up a unified frame
to tackle the stability problem of the K\"ahler Ricci flow
(c.f. Chen-Zheng \cite{zheng-2009krf}),
the Calabi flow and the pseudo-Calabi flow.
This method strongly relied on
the geometric structure of the space of K\"ahler metrics
.

First, we will prove the long time existence of the Calabi flow.
\begin{thm}\label{long}
On K\"ahler manifolds admitting an extremal metric $\omega$ and for any positive constant $\mathcal{K}, \lambda$, there is a
small constant $\epsilon$ depending on $\omega, \mathcal{K}, \lambda$, such that for any K\"ahler potential $\varphi_0$, 
if 
$$
|\varphi_0|_{C^{2,\alpha}(M)} < \mathcal{K}, \quad \lambda \omega < \omega_{\varphi_0}, \quad \int_M |\varphi_0|^2 \omega^n < \epsilon, 
$$
then the Calabi flow exists for all time.

\end{thm}
Next we want to study the modified Calabi flow. Let $K$ be a maximal compact subgroup in the reduced automorphism group. Denote the corresponding Lie algebra of $K$ by $h_0(M)$, which is the ideal of holomorphic vector fields with zeros.
For any holomorphic vector field $\tilde{Y} \in h_0(M)$, denote $\tilde{Y}=Y-\sqrt{-1}JY$. Then there is a real function $\theta_Y$ such that
$$
L_{\tilde{Y}}\omega =L_{Y}\omega
=\sqrt{-1}\partial\bar\partial\theta_{Y}(t)
$$
and
$$
\int_M \theta_Y \omega^n = 0.
$$

For an arbitrary metric 
$$
\omega_{\varphi} = \omega + \sqrt{-1} \partial \bar{\partial} \varphi ,
$$
the corresponding $\theta_{\tilde{Y}}(\varphi)$ is 
$$
\theta_{\tilde{Y}}(\varphi)=\theta_Y + \tilde{Y}(\varphi).
$$

Following Futaki-Mabuchi \cite{MR1314584}, suppose $\tilde{X}, \tilde{Y} \in h_0(M)$, then the bilinear form
$$B(\tilde{X},\tilde{Y})=\int_M\theta_{\tilde{X}(\varphi)}\theta_{\tilde{Y}(\varphi)}\omega_{\varphi}^n$$
 is independent of the choice
of $\omega_{\varphi}$ in the K\"ahler class $\Omega=[\omega]$.

Let $\varphi(t)$ be a one parameter of K\"ahler potentials satisfying the Calabi flow equation and let $\sigma(t)$ be the holomorphic group
generated by $X$, the real part of the extremal vector field $\tilde{X}$.
Then $\sigma^{\ast}(t)\omega
=\omega+i\partial\bar\partial\rho(t)$ satisfies the Calabi flow equation since
$$\frac{\partial\sigma^\ast(t)\omega}{\partial t}
=L_X\omega(t)
=i\partial\bar\partial (S(t)-\underline{S}).
$$
Hence we can choose $\rho(t)$ to be a parameter of K\"ahler potentials satisfying the Calabi flow equation starting from $0$.

Let $\psi(t)=\sigma(-t)^\ast(\varphi(t)-\rho(t))$. Notice that, by definition,
$X=\sigma^{-1}_\ast(\frac{\partial}{\partial t}\sigma)$.
So we obtain the modified Calabi flow,
\begin{align*}\label{modified calabi flow}
\frac{\partial \psi}{\partial t}
&=-X(\psi(t))+\sigma(-t)^\ast \left(\frac{\partial\varphi(t)}{\partial t}
-\frac{\partial\rho(t)}{\partial t} \right)\\
&=-X(\psi(t))+\sigma(-t)^\ast(S(\varphi(t))-\underline S)
- (S(\omega)-\underline S)\\
&=S_\psi-\underline S-\theta_X-X(\psi)\\
&=S_\psi-\underline S-\theta_X(\psi)
\end{align*}

\begin{thm}\label{exp}
On K\"ahler manifolds admitting an extremal metric $\omega$,
for any $\mathcal{K}$-invariant K\"ahler potential $\varphi_0$ close to $\omega$ (in the sense of Theorem (\ref{long})),
the modified Calabi flow exponentially converges to
a nearby extremal metric.
\end{thm}

\noindent {\bf Acknowledgements:}
The authors want to thank Professor X.X. Chen for many useful discussions on the stability problem of Calabi flow. They also want to express their gratitude to Professor V. Apostolov for his interest and help in this problem. The second author wants to thank his advisor Professor W.Y. Ding, also  
Professor X.H. Zhu and Professor H.Z. Li
for their help and encouragement. The first author would like to thank Professor P.F. Guan for valuable discussions.

\section{Long time existence}
First of all, we would like to give a rough estimate of the geodesic distance between any two K\"ahler potentials $\varphi_0, \varphi_1$ when
$$
\omega_{\varphi_0}, \omega_{\varphi_1} < \Lambda \omega.
$$

\begin{lem} 
\label{DE}
$$
d(\varphi_0, \varphi_1) < C(\Lambda) \left( \int_M |\varphi_0 - \varphi_1|^2 \omega^n \right)^{\frac{1}{2}}
$$
\end{lem}

\begin{proof}
Let $\varphi_t =(1-t)\varphi_0 + t \varphi_1$ for $0\leq t\leq1$.
Then
\begin{align*}
d(\varphi_0,\varphi_1)
&\leq L(\gamma_t)\\
&=\int_0^1 \left(\int_M \left(\frac{\partial\gamma_t}{\partial t} \right)^2
\omega^n_{\gamma_t} \right)^{\frac{1}{2}}dt\\
&\leq\int_0^1 \left(\int_M (\varphi_0 - \varphi_1)^2
\omega^n_{\gamma_t} \right)^{\frac{1}{2}}dt\\
&\leq C(\Lambda) \left( \int_M |\varphi_0 - \varphi_1|^2 \omega^n \right)^{\frac{1}{2}}.
\end{align*}

\end{proof}

We are ready to give a proof of \thmref{long}.
\begin{proof}
Suppose that the conclusion fails, then there exist positive constants $\mathcal{K}, \lambda, \Lambda$ and a sequence of
$\varphi^0_{s}$ such that
\[
|\varphi^0_s|_{C^{2,\alpha}} < \mathcal{K}, \ \lambda \omega < \omega_{\varphi_0^s} < \Lambda \omega, \ \int_M |\varphi^0_s|^2 \omega^n < \frac{1}{s} \quad s=1,2,3 \cdots .
\]
By virtue of the short time existence theorem,
we get a sequence of solutions $\varphi_s(t)$ satisfying the
flow equation \eqref{calabi flow}
with $\varphi_s(0)=\varphi^0_s$.
Let $T_s$ be the first time such that
$$|\varphi_s(T_s) - \rho(T_s) |_{c^{2,\alpha}(\omega(T_s))}=2C \quad \mbox{ holds}$$ 
or 
$$ \lambda \omega(T_s) < \omega_{\varphi_s(T_s)} < \Lambda \omega(T_s) \quad \mbox{ fails,}$$ 
where the constant $C$ is from \thmref{short time: stability}. Then $T_s$ is bounded from below for sufficiently large $s$. Otherwise there is a subsequence of $\varphi_s(T_s)$ converging to $\varphi_\infty$ in the $C^{2,\alpha'}(\omega)$ sense, where $\alpha'<\alpha$. Notice that $\lambda \omega \leq \omega_{\varphi_\infty} \leq \Lambda \omega $, but that $\lambda \omega < \omega_{\varphi_\infty} < \Lambda \omega $ fails. 

On the other hand, Lemma (\ref{DE}) shows that $d(0, \varphi_s^0) \rightarrow 0$ as $s \rightarrow \infty$. Since the distance function decreases under the Calabi flow, we have $$d(\rho(T_s), \varphi_s(T_s)) \rightarrow 0$$ as $s \rightarrow \infty$. Let $\varphi_\infty(t)$ be one parameter potentials satisfying the Calabi flow equation initiating from $\varphi_\infty$. Then for $t_0 \geq t$, 
\begin{eqnarray*}
d(\rho(t_0), \varphi_\infty(t_0)) &\leq& d(\rho(t), \varphi_\infty(t)) \\
&\leq& d(\rho(t), \rho(T_s)) + d(\rho(T_s),\varphi_s(T_s)) +  d(\varphi_s(T_s), \varphi_\infty(t)).
\end{eqnarray*}
By Lemma (\ref{DE}), $d(\varphi_s(T_s), \varphi_\infty(t)) \rightarrow 0$ as $s \rightarrow \infty$ and $t \rightarrow 0$. Hence $\rho(t_0) = \varphi_\infty(t_0)$, which implies $0 = \varphi_\infty$, a contradiction.

Moreover, from \thmref{regularity of PCF: reg of
PCF},
we obtain the higher order uniform bounds of the sequence of
the solutions:
\[
|\varphi_s(T_s) - \rho(T_s) |_{C^{k,\alpha}(\omega(T_s))} \leq C(k), \ \forall k\geq 0.
\]
Therefore we can choose a subsequence of $\phi_s=\sigma(-T_s)^*(\varphi_s(T_s) -\rho(T_s))$ so that
\[
\phi_s\rightarrow\phi_\infty\text{ in } C^{k,\alpha}(\omega),\forall k\geq 0,
\] and
\[
|\phi_\infty|_{C^{2,\alpha}(\omega)}=2C \quad \mbox{(or} \quad \lambda \omega < \omega_{\phi_\infty} < \Lambda \omega \ \mbox{fails).}
\]
However, this contradicts the fact that $d(0,\phi_\infty) = 0$.
\end{proof}

\begin{cor}
Given a K\"ahler potential $\varphi_0$ close to $0$ in the sense of \thmref{long}, then the modified Calabi flow stays in a neighborhood of $0$. If $\varphi_0$ is $\mathcal{K}$-invariant, then the modified Calabi flow converges to an extremal metric nearby. 
\end{cor}

\begin{proof}
That the modified Calabi flow stays in a neighborhood of $0$ can be easily seen from the regularity \thmref{regularity of PCF: reg of
PCF}. Notice that the Calabi flow decreases the Calabi energy, i.e.
$$
\frac{\partial}{\partial t} Ca(\omega_\varphi) = -2 \int_M \mathcal{L_\varphi}(S_\varphi) S_\varphi \ \omega_{\varphi}^n,
$$
where $\mathcal{L}_\varphi$ is the Lichnerowicz operator with respect to $\omega_\varphi$. It follows that we can take a sequence of $t_j \rightarrow \infty$ such that 
$$
\lim_{j \rightarrow \infty}\int_M \mathcal{L}_{\varphi(t_j)}(S_{\varphi(t_j)}) S_{\varphi(t_j)} \ \omega_{\varphi(t_j)}^n = 0.
$$
Then there is a subsequence of $t_j$ such that $\psi(t_j)$ converges to a potential 
$\psi_\infty$ in $C^{\infty}$ and
$$
\int_M \mathcal{L_{\psi_\infty}}(S_{\psi_\infty}) S_{\psi_\infty} \ \omega_{\psi_\infty}^n = 0.
$$
Hence $\omega_{\psi_\infty}$ is an extremal metric. If $\varphi_0$ is $\mathcal{K}$-invariant, then $\psi_\infty$ is a fixed point under the modified Calabi flow and the modified Calabi flow decreases the geodesic distance between $\psi(t)$ and $\psi_\infty$. Hence the flow converges to $\psi_\infty$.
\end{proof}

\section{Exponential decay}
We define the modified Calabi energy as
\begin{align*}
\widetilde{Ca}(\psi) = \int_M(S(\psi)-\underline{S}
-\theta_{X}(\psi))^2\omega_\psi^n.
\end{align*}
The evolution of the modified Calabi energy along the modified Calabi flow is
\begin{align*}
\partial_t\int_M\dot\psi^2\omega_\psi^n
&=\int_M(2\dot\psi\ddot\psi
+\dot\psi^2\triangle_\psi\dot\psi)\omega_\psi^n\\
&=2\int_M\dot\psi(\dot S_\psi-\dot\theta_X(t)
-\dot\psi_i\dot\psi^i)\omega_\psi^n\\
&=2\int_M\dot\psi(-L\dot\psi+\dot\psi^iS_i-\dot\theta_X(t)
-\dot\psi_i\dot\psi^i)\omega_\psi^n\\
&=2\int_M\dot\psi(-L\dot\psi+\dot\psi^i(\dot\psi_i+\theta_X(t)_i)
-\dot\theta_X(t)-\dot\psi_i\dot\psi^i)\omega_\psi^n\\
&=2\int_M\dot\psi(-L\dot\psi+\dot\psi^i\dot\psi_i+X(\dot\psi)
-X(\dot\psi)-\dot\psi_i\dot\psi^i)\omega_\psi^n\\
&=-2\int_M\dot\psi L\dot\psi\omega_\psi^n.
\end{align*}
In this computation we use the identities
$$
\dot\psi^i({\theta_X}_i+(X(\psi))_i)=\dot\psi^i(\theta_X(t))_i=X(\dot\psi).
$$

The modified Calabi-Futaki invariant
\begin{align}\label{futaki}
\tilde{F}(\tilde{Y})&=F(\tilde{Y})-B(\tilde{X},\tilde{Y})
=\int_M\theta_{\tilde{Y}}(\psi)
(S-\underline{S}-\theta_{\tilde{X}}(\psi))\omega_\psi^n
\end{align}
is equal to zero when $M$ admits an extremal metric $\omega$. This shows that the direction of the modified Calabi flow is always perpendicular to the kernel of the Lichnerowicz operator. To obtain exponential convergence, one needs to give a uniform lower bound of the the first eigenvalue of $L_t$ along the modified Calabi flow. 
More precisely, we have the following lemma which is similar to Chen-Li-Wang \cite{MR2481736}.
\begin{lem}
Along the modified Calabi flow, there is a positive constant $\lambda > 0$ such that for sufficiently large $t$ and for any
\begin{align*}
f\in A_t&=\{f\in C_R^\infty(M) \vert \int_Mf\omega^n_{\psi(t)}=0 \ \mbox{and} \ 
\int_M \theta_{Y}(t)f\omega_{\psi(t)}^n
=0,
\forall \tilde{Y}\in h_0(M)\},
\end{align*}
we have
$$
\int_M L_t(f) f \omega^n_{\psi(t)}
\geq\lambda\int_M f^2\omega^n_{\psi(t)}.
$$
\end{lem}
\begin{proof}
If not, there must be a sequence $\psi_s=\psi(s)$ and $f_s$ such that
\begin{align}\label{epd: assu}
\int_M|(f_s)_{ij}|^2\omega^n_{\psi_s}
<\frac{1}{s};
\int_Mf_s^2\omega^n_{\psi_s}=1;
\int_Mf_s\omega^n_{\psi_s}=0.
\end{align}
Since the $C^l$ norm of $\psi_s$ is uniformly bounded for any $l\geq0$.
Using the Ricci identity
$$\int_M|(f_s)_{i\bar j}|^2\omega_{\psi_s}^n
=\int_M|(f_s)_{ij}|^2\omega_{\psi_s}^n
+\int_MR^{i\bar{j}}(f_s)_i(f_s)_{\bar j}\omega_{\psi_s}^n$$
and the interpolation inequality,
we conclude that $f_s$ are uniformly $W^{2,2}$ bounded.
So we can pass to the limit and get
\begin{align}\label{epd: assu infty}
\int_M|(f_\infty)_{ij}|^2\omega^n_{\psi_\infty}=0;
\int_Mf_\infty^2\omega^n_{\psi_\infty}=1;
\int_Mf_\infty\omega^n_{\psi_\infty}=0.
\end{align}
Since in local coordinates, $\uparrow\bar\partial f_\infty$ is holomorphic in the weak sense, $f_\infty$ is smooth indeed.
From the assumption of $A_t$ we have
\begin{align*}
\int_M \theta_{Y}(\psi_\infty)f_\infty\omega_{\psi_\infty}^n
=0,
\forall \tilde{Y}\in h_0(M).
\end{align*}
In particular, we may choose
$\tilde{Y}=\uparrow\bar\partial f_\infty\in h_0(M)$.
Hence,
\begin{align*}
\int_Mf_\infty^2\omega_\infty^n=0.
\end{align*}
This contradicts \eqref{epd: assu infty}.
\end{proof}
It is easy to see that
$
\widetilde{Ca}(\psi(t))
\leq C e^{-\lambda t}
$. 
To get exponential convergence of $\psi(t)$, we calculate the evolution formula for $\int_M |\nabla^k (\psi(t)-\psi_\infty)|^2 \ \omega^n$ :
\begin{eqnarray*}
& & \frac{\partial}{\partial t} \int_M |\nabla^k (\psi(t)-\psi_\infty) |^2 \ \omega^n \\
&=& \int_M \nabla^k (S-\underline{S} -\theta_X(\psi)) * \nabla^k (\psi(t)-\psi_\infty) \ \omega^n\\
&=& \int_M (S-\underline{S} -\theta_X(\psi)) * \nabla^{2k} (\psi(t)-\psi_\infty) \ \omega^n \\
&\leq& \left(\int_M (S-\underline{S}-\theta(X))^2 \ \omega^n \right)^{1/2} \left(\int_M |\nabla^{2k} (\psi(t)-\psi_\infty)|^2 \ \omega^n\right)^{1/2}\\
&\leq& C \parallel S-\underline{S}-\theta(X) \parallel_{L^2(\omega)} \\
&\leq& C \parallel S-\underline{S}-\theta(X) \parallel_{L^2(\omega_t)} \\
&\leq& C e^{-\lambda_2 t}.
\end{eqnarray*}

By the Sobolev embedding, we conclude that

$$
\parallel \psi_t - \psi_{\infty} \parallel_{C^l(\omega)} \ \leq \ \parallel \psi_t - \psi_{\infty} \parallel_{W^{k,2}(\omega)} \ \leq \ C e^{-\lambda_2 t}.
$$

Hence we obtain the result stated in \thmref{exp}.

Hongnian Huang

Centre interuniversitaire de recherches en geometri et topologie

Universite du Quebec a Montr茅al

Case postale 8888, Succursale centre-ville

Montreal (Quebec), Canada

E-mail address: hnhuang@gmail.com

\vspace{0.2 in}
Kai Zheng

Academy of Mathematics and Systems Sciences 

Chinese Academy of Sciences

Bei- jing, 100190, P.R. China.

E-mail address: kaizheng@amss.ac.cn


\begin{thebibliography}{10}

\bibitem{MR1314584}
A. Futaki\ and\ T. Mabuchi, Bilinear forms and extremal K\"ahler vector fields associated with K\"ahler classes, Math. Ann. {\bf 301} (1995), no.~2, 199--210.

\bibitem{MR645743}
E. Calabi, Extremal K\"ahler metrics, in {\it Seminar on Differential Geometry}, 259--290, Ann. of Math. Stud., 102, Princeton Univ. Press, Princeton, N.J.

\bibitem{MR780039}
E. Calabi, Extremal K\"ahler metrics. II, in {\it Differential geometry and complex analysis}, 95--114, Springer, Berlin.

\bibitem{MR1969662}
E. Calabi\ and\ X. X. Chen, The space of K\"ahler metrics. II, J. Differential Geom. {\bf 61} (2002), no.~2, 173--193.

\bibitem{MR1863016}
X. X. Chen, The space of K\"ahler metrics, J. Differential Geom. {\bf 56} (2000), no.~2, 189--234.

\bibitem{MR1820328}
X. X. Chen, Calabi flow in Riemann surfaces revisited: a new point of view, Internat. Math. Res. Notices {\bf 2001}, no.~6, 275--297.

\bibitem{zheng-2009}
X. X. Chen, W. Y. Ding\ and\ K. Zheng, Pseudo-Calabi flow. Unpublished, 2009.

\bibitem{MR2405167}
X. X. Chen\ and\ W. Y. He, On the Calabi flow, Amer. J. Math. {\bf 130} (2008), no.~2, 539--570.

\bibitem{chen-2007}
X. X. Chen\ and\ W. Y. He,
The Calabi flow on K\"ahler surface with bounded Sobolev constant--({I}), arXiv:0710.5159, 2007.

\bibitem{chen-2008}
X. X. Chen\ and\ W. Y. He,
The Calabi flow on toric Fano surface, arXiv:0807.3984, 2008.

\bibitem{MR2481736}
X. X. Chen, H. Z. Li\ and\ B. Wang, K\"ahler-Ricci flow with small initial energy. Geom. Funct. Anal. 18 (2009), no. 5, 1525--1563.

\bibitem{MR1817370}
X. X. Chen\ and\ G. Tian, Ricci flow on K\"ahler manifolds, C. R. Acad. Sci. Paris S\'er. I Math. {\bf 332} (2001), no.~3, 245--248.

\bibitem{MR1893004}
X. X. Chen\ and\ G. Tian, Ricci flow on K\"ahler-Einstein surfaces, Invent. Math. {\bf 147} (2002), no.~3, 487--544.

\bibitem{MR2121892}
X. X. Chen\ and\ G. Tian, Uniqueness of extremal K\"ahler metrics, C. R. Math. Acad. Sci. Paris {\bf 340} (2005), no.~4, 287--290.

\bibitem{zheng-2009krf}
X. X. Chen\ and\ K. Zheng, K\"ahler {R}icci flow in the space of {K}\"aler metric. Unpublished, 2009.

\bibitem{chenzhu}
X. X. Chen\ and\ M. J. Zhu,
Liouville energy on a topological two sphere, arXiv:0710.4320, 2007.

\bibitem{MR1101689}
P. T. Chru\'sciel, Semi-global existence and convergence of solutions of the Robinson-Trautman ($2$-dimensional Calabi) equation, Comm. Math. Phys. {\bf 137} (1991), no.~2, 289--313.

\bibitem{MR1736211}
S. K. Donaldson, Symmetric spaces, K\"ahler geometry and Hamiltonian dynamics, in {\it Northern California Symplectic Geometry Seminar}, 13--33, Amer. Math. Soc. Transl. Ser. 2, 196, Amer. Math. Soc., Providence, RI.

\bibitem{MR2103718}
S. K. Donaldson, Conjectures in K\"ahler geometry, in {\it Strings and geometry}, 71--78, Amer. Math. Soc., Providence, RI.

\bibitem{he-2009}
W. Y. He, Local solution and extension to the Calabi flow, arXiv:0904.0978, 2009.

\bibitem{MR909015}
T. Mabuchi, Some symplectic geometry on compact K\"ahler manifolds. I, Osaka J. Math. {\bf 24} (1987), no.~2, 227--252.

\bibitem{MR1165352}
S. Semmes, Complex Monge-Amp\`ere and symplectic manifolds, Amer. J. Math. {\bf 114} (1992), no.~3, 495--550.

\bibitem{MR1991140}
M. Struwe. Curvature flows on surfaces. Ann. Sc. Norm. Super. Pisa Cl. Sci. (5) 1 (2002), no. 2, 247--274.

\bibitem{MR2357473}
V. Tosatti\ and\ B. Weinkove, The Calabi flow with small initial energy, Math. Res. Lett. {\bf 14} (2007), no.~6, 1033--1039.


\end{thebibliography}
\end{document}